\documentclass{article}
\usepackage{amsthm,amsmath,amssymb,amsfonts,stmaryrd}
  \newtheorem{thm}{Theorem}[section]
  \newtheorem{lem}[thm]{Lemma}
  \newtheorem{prop}[thm]{Proposition}
  \newtheorem{cor}[thm]{Corollary}
  \theoremstyle{definition}
  \newtheorem{defn}[thm]{Definition}
  \newtheorem{exm}[thm]{Example}
  \newtheorem{rmk}[thm]{Remark}
  
 \newcommand\ra{\rightarrow}
 
  \newcommand\gam{\gamma}

 \newcommand\s{\subseteq}

 \numberwithin{equation}{section}

\begin{document}

\title{\bf Some results on $L$-complete lattices }
\author{\bf Anatolij Dvure\v{c}enskij$^{1,2}$, Omid Zahiri$^{3}$\footnote{Corresponding author: O. Zahiri }\\
 {\small\em $^1$ Mathematical Institute,  Slovak Academy of Sciences, \v Stef\'anikova 49, SK-814 73 Bratislava, Slovakia} \\
{\small\em $^2$ Depart. Algebra  Geom.,  Palack\'{y} Univer.17. listopadu 12,
CZ-771 46 Olomouc, Czech Republic} \\
{\small\em  $^3$University of Applied Science and Technology, Tehran, Iran}\\
{\small\em dvurecen@mat.savba.sk\quad  om.zahiri@gmail.com} \\
}
\date{}
\maketitle
\begin{abstract}
The paper deals with special types of $L$-ordered set, $L$-fuzzy complete lattices, and fuzzy directed complete posets (fuzzy $dcpo$s).
First, a theorem for constructing monotone maps is proved, a characterization for monotone maps on an $L$-fuzzy
complete lattice is obtained, and
it is proved that if $f$ is a monotone map on an $L$-fuzzy complete lattice $(P;e)$, then $\sqcap S_f$ is
the least fixpoint of $f$. A relation between $L$-fuzzy complete lattices and fixpoints is found and
fuzzy versions of monotonicity, rolling, fusion  and exchange rules on $L$-complete lattices are stated.
Finally, we investigate $Hom(P,P)$, where $(P;e)$ is a fuzzy $dcpo$, and we
show that $Hom(P,P)$ is a fuzzy $dcpo$, the map $\gam\mapsto \bigwedge_{x\in P}e(x,\gam(x))$ is a
fuzzy directed subset of $Hom(P,P)$, and we investigate its join.
\end{abstract}



{\small {\it Acknowledgement:} This work was supported by  the Slovak Research and Development Agency under contract APVV-0178-11 and  grant VEGA No. 2/0059/12 SAV, and
CZ.1.07/2.3.00/20.0051.}

\section{ Introduction }
Fixed point theory serves as an essential tool for various branches of mathematical
analysis and its applications. There are three main approaches to this
theory. The first one is the metric approach in which one makes use of the metric properties
of the underlying spaces and self-maps. (A primary example of this approach is Banach's
Contraction Mapping Theorem.) The second approach is the topological one in which
one utilizes the topological properties of the underlying spaces and continuity of self-maps. (A primary example of this approach is Brouwer's Fixed Point Theorem.) Finally,
the third approach is the order-theoretic approach.

Recently,  based on complete Heyting algebras and fuzzy $L$-order relation, Zhang and Xie \cite{ZXF} have defined and studied
$L$-fuzzy complete lattices, which are generalizations of traditional
complete lattices. They discussed their properties, showed  that they coincide with complete and co-complete categories
enriched over the frame $L$ \cite{Wagner}, and they proved the Tarski Fixed-Point Theorem
for an $L$-fuzzy complete lattice.

Using complete Heyting algebras, Fan and Zhang \cite{Fan,ZF}
studied quantitative domains through fuzzy
set theory. Their approach first defines a fuzzy partial order, specifically a degree function, on a non-empty set.
Then they define and study fuzzy directed subsets and (continuous) fuzzy directed complete posets ($dcpo$s for short).
Moreover, Yao \cite{YSI} and   Yao and Shi \cite{Y2S} pursued an investigation on quantitative domains via fuzzy sets. They defined the notions of
fuzzy Scott topology on fuzzy $dcpo$s, Scott convergence and topological convergence for stratified $L$-filters and study them. They
showed that the category of fuzzy $dcpo$s with fuzzy Scott continuous maps is Cartesian-closed.

In \cite{ZL}, Zhang and Liu defined a kind of an $L$-frame by a pair $(A,i_A)$, where A is a classical frame and
$i_A:L\ra A$ is a frame morphism. For a stratified $L$-topological space $(X,\delta)$, the pair $(\delta,i_X)$
is one of this kind of $L$-frames, where $i_X:L\ra \delta$, is a map which sends $a\in L$ to the constant map with the value $a$.
Conversely, a point of an $L$-frame $(A,i_A)$ is a frame morphism $p:(A,i_A)\ra (L,id_L)$ satisfying $p\circ i_A=id_L$ and
$Lpt(A)$ denotes the set of all points of $(A,i_A)$. Then $\{\Phi_x:Lpt(A)\ra L|\ \forall\, p\in Lpt(A),\ \Phi_x(p)=p(x)\}$
is a stratified $L$-topology on $Lpt(A)$. By these two assignments, Zhang and Liu constructed an adjunction between
$SL-Top$ and $L-Loc$ and consequently they established the Stone Representation Theorem for distributive lattices by means of this adjunction.
They pointed out that, from the viewpoint of lattice theory, Rodabaugh's fuzzy version of the Stone
representation theory is just one and it has nothing different from the classical one.
While in our opinion, Zhang-Liu's $L$-frames
preserve many features and also seem to have no strong difference from a crisp one.

In \cite{Y2}, Yao introduced an $L$-frame by an $L$-ordered set equipped with some further conditions. It is a
complete $L$-ordered set with the meet operation having a right fuzzy adjoint. They established an
adjunction between the category of stratified $L$-topological spaces and the category of $L$-locales,
the opposite category of this kind of $L$-frames.  Moreover, Yao and Shi, \cite{Yao-Shi},
defined on fuzzy $dcpo$s an $L$-topology, called the fuzzy Scott topology, and then they studied its
properties. They defined Scott convergence, topological convergence for stratified $L$-filters and showed that
a fuzzy $dcpo$ is continuous if and only if, for any stratified $L$-filter, the fuzzy Scott convergence coincides with convergence
with respect to the fuzzy Scott topology.

In the mid-1950's Tarski \cite{TA} published an interesting result: Every complete lattice has the fixed point property, that is, every order preserving mapping has a fixpoint.
Davis \cite{DA} proved the converse: Every lattice with the fixed point property is complete.
Tarski's Fixpoint Theorem generalizes to $CPO$ (an abbreviation for a $\vee$-complete poset with a bottom element), i.e. if $f:P\ra P$ is an order preserving map and $P$ is a $CPO$, then
the set of fixpoints of $f$,  $Fix(f)$, is a $CPO$ and so $Fix(f)$ has a least element. But, what is the relationship
between the least element of $f$ and other
points of $P$? The main purpose of this paper is to answer to this question.

The present paper is organized as follows. In Section 2, we list some preliminary notions and results that will be used in the paper.
In Section 3, we consider $L$-fuzzy complete lattices. We show that if $f$ is a monotone map on an $L$-fuzzy complete
lattice $(P;e)$, then $\sqcap S_f$ and $\sqcup T_f$ are the least and greatest fixpoint of $f$ and so we find a relation
between these elements and other points of $P$. Also, we show that
every $L$-complete lattice is, up to isomorphism, an $L$-complete lattice of fixpoints and we propose
fuzzy versions of monotonicity, rolling, fusion  and exchange rules on $L$-complete lattices.
In Section 4, we define the concept of a $t$-fixpoint and we prove that if $(P;e)$ is a fuzzy $dcpo$, then $H_P$, the set of monotone maps on $(P;e)$, is a fuzzy $dcpo$. We use it to find some of $t$-fixpoints of $f$. Finally, we find  conditions
under which $\sqcap S_f$ exists.

\section{ Preliminaries}

We start with some notions from \cite{davey, Fan}. A non-empty subset $D$ of a poset $(P;\leq)$ is called {\it directed} if, for each pair of elements $x,y\in D$, there exists $z\in D$ such that
$x,y\leq z$. We say that a poset $(P;\leq)$ is a {\it pre}-$COP$ or a $dcpo$ (an abbreviation for a directed complete partially ordered set) if, for each directed subset $D$ of $P$, the join of $D$, $\bigvee D$ (the least upper bound of $D$ in $L$) exists. A $dcpo$ $(P;\leq)$ is called a $CPO$ (an abbreviation for a complete partially ordered set) if
$P$ has a bottom element.

Let $(L;\vee,\wedge,0,1)$  be a bounded lattice. For $a,b\in L$, we say that $c\in L$ is a {\it relative pseudocomplement} of $a$
with respect to $b$ if $c$ is the largest element with $a\wedge c\leq b$ and we denote it by $a\ra b$.
A lattice $(L;\vee,\wedge)$ is said to be a {\it Heyting algebra} if the relative pseudocomplement $a\to b$ exists for all elements $a,b \in L$.
A {\it frame} is a complete lattice $(L;\vee,\wedge)$ satisfying the infinite distributive law $a\wedge \bigvee S=\bigvee_{s\in S}(a\wedge s)$ for
every $a\in L$ and $S\subseteq L$. It is well known that $L$ is a frame if and only if it is a complete Heyting algebra.
In fact, if $(L;\vee,\wedge)$ is a frame, then for each $a,b\in L$, the relative pseudocomplement of $a$ with respect to $b$,  is the element
$a\ra b:=\vee\{x\in L|\ a\wedge x\leq b\}$. In the following, we list some important properties of complete Heyting algebras,
for more details relevant to frames and Heyting algebras, we refer to \cite{Johnston} and \cite[Section 7]{Blyth}:

\begin{enumerate}
\item[(i)]\ $(x\wedge y)\ra z=x\ra (y\ra z)$;
\item[(ii)]\ $x\ra (\bigwedge Y)=\bigwedge_{y\in Y}(x\ra y)$;
\item[(iii)]\ $(\bigvee Y)\ra z=\bigwedge_{y\in Y}(y\ra z)$.
\end{enumerate}

From now on, in this paper, $(L;\vee,\wedge,0,1)$ or simply $L$ always denotes a frame and $L^X$ denotes the set of
all maps from a set $X$ into $L$.

\begin{defn}\cite{Bel1,Bel2,ZF,ZF2}
Let $P$ be a set and $e : P\times P\ra L$ be a map. The pair $(P;e)$ is called an {\it $L$-ordered set} if, for all $x,y,z\in P$, we have
\begin{itemize}
 \item[{\rm(E1)}] $e(x,x)=1$;
 \item[{\rm(E2)}] $e(x,y)\wedge e(y,z)\leq e(x,z)$;
 \item[{\rm(E3)}] $e(x,y)=e(y,x)=1$ implies $x=y$.
\end{itemize}

\begin{prop}{\rm \cite[Prop. 3.7]{YSI}}\label{YSI Prop 3.7}
Let $(X;e)$ be an $L$-ordered set. Then for each $x,y\in X$,
$$e(x,y)=\bigwedge_{z\in X}\big(e(z,x)\ra e(z,y) \big)=\bigwedge_{z\in X}\big(e(y,z)\ra e(x,z) \big).$$
\end{prop}

In an $L$-ordered set $(P;e)$, the map $e$ is called an {\it $L$-order relation} on $P$.
If $(P;\leq)$ is a classical poset, then $(P;\chi_{\leq})$ is an $L$-ordered set, where $\chi_{\leq}$ is the
characteristic function of $\leq$. We usually denote this $L$-ordered set by $(P;e_\leq)$. Moreover, for each $L$-ordered set $(P;e)$,
the set $\leq_e=\{(x,y)\in P\times P|\ e(x,y)=1\}$
is a crisp partial order on $P$ and $(P;\leq_e)$ (if there is no ambiguity we write $(P;\leq)$) is a poset. Assume that $(P;e)$ is an $L$-ordered set and $\phi\in L^P$.
Define $\downarrow \phi\in L^P$ and $\uparrow \phi\in L^P$, see \cite{Fuentes,ZF}, as follows:
 $$\downarrow\phi(x)=\bigvee_{x'\in P}(\phi(x')\wedge e(x,x')), \quad \quad \uparrow\phi(x)
 =\bigvee_{x'\in P}
 (\phi(x')\wedge e(x',x)),\quad \forall x\in P.$$
\end{defn}

\begin{defn}\cite{Y2,YSI}
 A map $f:(P;e_P)\ra (Q;e_Q)$ between two $L$-ordered sets is called {\it monotone} if for all $x,y\in P$,
$e_P(x,y)\leq e_Q(f(x),f(y))$.
\end{defn}
A monotone map $f:(P;e_P)\ra (Q;e_Q)$ is called an {\it $L$-order isomorphism} if $f$ is one-to-one,
onto and $e_P(x,y)=e_Q(f(x),f(y))$ for all $x,y\in P$.

\begin{defn}\cite{YL,YSI}
Let $(P;e_P)$ and $(Q;e_Q)$ be two $L$-ordered sets and $f:P\ra Q$ and $g:Q\ra P$ be two monotone maps. The pair
$(f,g)$ is called a {\it fuzzy Galois connection} between $P$ and $Q$ if $e_Q(f(x),y)=e_P(x,g(y))$
for all $x\in P$ and $y\in Q$, where $f$ is called the {\it fuzzy left adjoint} of $g$, and dually, $g$ is called the {\it fuzzy right adjoint} of $f$.
\end{defn}

\begin{thm}\label{Galois}{\rm \cite[Thm. 3.2]{YL}}
A pair
$(f,g)$ is a fuzzy Galois connection on $(X,e_X)$ and $(Y,e_Y)$ if and only if both $f$ and $g$ are monotone and
$(f,g)$ is a (crisp) Galois connection on $(X,\leq_{e_X})$ and $(Y,\leq_{e_Y})$.
\end{thm}

\begin{defn}\cite{YSI,YL}
Let $(P;e)$ be an $L$-ordered set and $S\in L^P$. An element $x_0$ is called a {\it join} (respectively, a {\it meet}) of $S$, in symbols $x_0=\sqcup S$
(respectively, $x_0=\sqcap S$)  if, for all $x\in P$,
\begin{itemize}
\item[{\rm (J1)}] $S(x)\leq e(x,x_0)$ (respectively, (M1), $S(x)\leq e(x_0,x))$;
\item[{\rm (J2)}] $\bigwedge_{y\in P} (S(y)\ra e(y,x))\leq e(x_0,x)$ (respectively, (M2),
$\bigwedge_{y\in P} (S(y)\ra e(x,y))\leq e(x,x_0))$.
\end{itemize}
If a join and a meet of $S$ exist, then they are unique (see \cite{ZF}).
\end{defn}

\begin{thm}\label{join and meet}{\em\cite[Thm. 2.2]{ZF}}
Let $(P;e)$ be an $L$-ordered set, $x_0\in P$, and $S\in L^P$. Then
\begin{itemize}
\item[{\rm (i)}] $x_0=\sqcup S$ if and only if $e(x_0,x)=\bigwedge_{y\in P} (S(y)\ra e(y,x))$ for all $x\in P$;
\item[{\rm (ii)}] $x_0=\sqcap S$ if and only if $e(x,x_0)=\bigwedge_{y\in P} (S(y)\ra  e(x,y))$ for all $x\in P$.
\end{itemize}
\end{thm}

\begin{defn}\label{down and up}
Let $(P;e)$ be an $L$-ordered set. For all $a\in P$, $\downarrow a:P\ra L$ and $\uparrow a:P\ra L$  are defined by $\downarrow a(x)=e(x,a)$ and $\uparrow a(x)=e(a,x)$, respectively.
It can be easily shown that $\sqcup\downarrow a=a$ and $\sqcap\uparrow a=a$ for all $a\in P$ (see \cite[Prop 3.16]{YSI}).
\end{defn}

In \cite{ZXF}, Zhang  et al. an $L$-fuzzy complete lattice was introduced:
An $L$-ordered set $(P;e)$ is called an $L$-{\it fuzzy complete lattice} (or an $L$-{\it complete lattice}, for short)
if, for all $S\in L^P$, $\sqcap S$
and $\sqcup S$ exist. If $(P;e)$ is an $L$-complete lattice,
then $(P;\leq_e)$ is a complete lattice,
where $\vee S=\sqcup \chi_S$ and $\wedge S=\sqcap\chi_S$ for any $S\s P$.

\begin{thm}\label{thm: complete} {\rm \cite[Thm. 2.20]{ZXF} }
Let $X$ be a non-empty set. Then $(L^X;\tilde{e})$ is an $L$-complete lattice, where $\tilde{e}(f,g)=\bigwedge_{x\in X}
(f(x)\ra g(x))$ for all $f,g\in L^X$.
\end{thm}

Suppose that $X$ and $Y$ are two sets. For each mapping $f:X\ra Y$,
we have a mapping $f^{\ra}:L^X\ra L^Y$, defined by
$$(\forall y\in Y)(\forall A\in L^X) \Big(f^{\ra}(A)(y)=\bigvee\{A(x)|\ x\in X, \ f(x)=y\}\Big).$$
For simplicity, for any $A\in L^X$, we use $f(A)$ instead of $f^{\ra}(A)$.

\begin{defn}\cite{YSI}
Let $(P;e)$ be an $L$-ordered set. An element $x\in P$ is called a {\it maximal (or minimal)}
element of $A\in L^P$, in symbols $x=\max A$ (or $x=\min A$), if $A(x)=1$ and for all $y\in P$,
$A(y)\leq e(y,x)$ (or $A(y)\leq e(x,y)$). It is easy to see that if A has a maximal (or minimal)
element, then it is unique.
\end{defn}

\begin{defn}\cite{LaiZhang,YSI}
Let $(X;e)$ be an $L$-ordered set. An element $D\in L^X$ is called a {\it fuzzy directed subset} of $(P;e)$ if
\begin{itemize}
\item[{\rm (FD1)}] $\bigvee_{x\in X} D(x)=1$;
\item[{\rm (FD2)}] for all $x,y\in X$, $D(x)\wedge D(y)\leq \bigvee_{z\in X} \big(D(z)\wedge e(x,z)\wedge e(y,z)\big)$.
\end{itemize}
An $L$-ordered set $(X;e)$ is called a {\it fuzzy dcpo}
if every fuzzy directed subset of $(X;e)$ has a join.
\end{defn}

\section{ {\bf Fixpoints on $L$-complete lattices} }

In this section, monotone maps on $L$-ordered set play an important role. Thus, in Theorem \ref{Example}, we propose a
procedure for constructing monotone maps. Then we show that every $L$-complete lattice is, up to isomorphism,
an $L$-complete lattice of fixpoints. We find a relation between least and greatest fixpoints of a monotone map
$f$ on an $L$-complete lattice $(P;e)$ with join and meet with some special elements of $L^P$. Finally, we present a fuzzy version of monotonicity, rolling, fusion  and exchange rules on $L$-complete lattices.

\begin{lem}\label{lem: 1}
Let $(L;\vee,\wedge)$ be a frame and $a_i,b_i\in L$ for all $i\in I$. Then
$$\bigwedge_{i\in I}(a_i\ra b_i)\leq (\bigvee_{i\in I}a_i)\ra (\bigvee_{i\in I}b_i).$$
\end{lem}

\begin{proof}
Let $u:=\bigwedge_{i\in I}(a_i\ra b_i)$. Then $u\leq a_i\ra b_i$ for all $i\in I$, so $u\wedge a_i\leq b_i$ for all $i\in I$.
It follows that $u\wedge \bigvee_{i\in I} a_i=\bigvee_{i\in I}(u\wedge a_i)\leq \bigvee_{i\in I}b_i$ and hence
$u\leq (\bigvee_{i\in I}a_i)\ra (\bigvee_{i\in I}b_i)$.
\end{proof}

In Theorem \ref{Example} and  Corollary \ref{cor: 3.2}, we will use the following proposition that has been stated by W. Yao \cite[Prop. 3.16]{YSI}.
We did not find a proof for it, also we could not prove it. In Proposition \ref{modify}, we will prove the proposition with an additional condition (indeed, we add ``$\sqcup S$ {\rm (resp. $\sqcap S$)} exist'' 
in the statement of \cite[Prop. 3.16]{YSI}). 

\begin{prop}\label{modify}
Let $(P;e)$ be an $L$-ordered set, $S\in L^P$ and $\sqcup S$ {\rm (resp. $\sqcap S$)} exist.
Then $a=\max S$ {\rm(resp. $a'=\min S$)} if and only if $S(a)=1$ and
$a=\sqcup S$ {\rm (resp. $S(a')=1$ and $a'=\sqcap S$)}.
\end{prop}

\begin{proof}
Let $a=\max S$ and $b=\sqcup S$ for some $b\in P$. Then
$S(a)=1$ and $S(y)\leq e(y,a)$ for all $y\in P$.
By (J1), $1=S(a)\leq e(a,b)$. Also, by (J2),
$e(b,a)=\bigwedge_{y\in P}\big(S(y)\ra e(y,a) \big)=\bigwedge_{y\in P}1=1$, so $a=b$. The other part can be proved in a similar way.
\end{proof}

\begin{thm}\label{Example}
Let $(P;e)$ be an $L$-ordered set, $(Q;e')$ be an $L$-complete lattice and $f:P\ra Q$ be a map.
For each $x\in P$, we define $S_x:Q\ra L$, by $S_x(y)=\bigvee_{\{z\in P| f(z)=y\}}e(z,x)$. Let
$F:P\ra Q$ be defined by $F(a)=\sqcup S_a$ for all $a\in P$. Then $F$ is monotone.
Moreover, $f$ is monotone if and only if
$f=F$.
\end{thm}

\begin{proof}
First we show that $F$ is monotone, that is $e(a,b)\leq e'(F(a),F(b))$ for all $a,b\in P$.
Put $a,b\in P$. Set $u_a=\sqcup S_a$ and $u_b=\sqcup S_b$. By Theorem \ref{join and meet}, for all $x\in Q$,
we have
$$e'(u_a,x)=\bigwedge_{y\in Q}\big(S_a(y)\ra e'(y,x) \big),\quad   e'(u_b,x)=\bigwedge_{y\in Q}\big(S_b(y)\ra e'(y,x) \big).$$
Hence,
\begin{eqnarray}
\label{Example 1}e'(u_a,u_b)=\bigwedge_{y\in Q}\big(S_a(y)\ra e'(y,u_b) \big)\\
\label{Example 2} S_b(y)\leq e'(y,u_b) \mbox{ for all $y\in Q$. }
\end{eqnarray}
From (\ref{Example 2}), it follows that
$S_a(y)\ra S_b(y)\leq S_a(y)\ra e'(y,u_b)$ for all $y\in Q$ and so by (\ref{Example 1}),
$\bigwedge_{y\in Q}\big(S_a(y)\ra S_b(y)\big) \leq \bigwedge_{y\in Q}\big( S_a(y)\ra e'(y,u_b) \big)=e'(u_a,u_b)$.
Also,
\begin{eqnarray*}
\bigwedge_{y\in Q}\big(S_a(y)\ra S_b(y)\big)&=&\bigwedge_{y\in f(P)}\big(S_a(y)\ra S_b(y)\big), \mbox{ since $S(a)(y)=0$, for all $y\in Q-f(P)$} \\
&=& \bigwedge_{y\in P}\big(S_a(f(y))\ra S_b(f(y))\big)\\
&=& \bigwedge_{y\in P}\Big( (\bigvee_{f(z)=f(y)}e(z,a))\ra (\bigvee_{f(w)=f(y)}e(w,a)) \Big)\\
&\geq & \bigwedge_{y\in P} \Big( \bigwedge_{f(z)=f(y)} (e(z,a)\ra e(z,b))\Big), \mbox{ by Lemma \ref{lem: 1} } \\
&=& \bigwedge_{z\in P} \big( e(z,a)\ra e(z,b) \big) \\
&=& e(a,b), \mbox{ by Proposition \ref{YSI Prop 3.7}}
\end{eqnarray*}
so $F:(P;e)\ra (Q;e')$ is monotone. Now, we show that, if $f$ is monotone, then $f=F$. Suppose that $f:(P;e)\ra (Q;e')$
is monotone. Put $a\in P$. For all $y\in Q$, $S_a(y)=\bigvee_{\{z\in P|\ f(z)=y\}}e(z,a)$, so
$S_a(f(a))=1$. Since $f$ is monotone, $f(z)\in y$ implies that $e(z,a)\leq e'(f(z),f(a))=e'(y,f(a))$, hence
$S_a(y)=\bigvee_{\{z\in P|\ f(z)=y\}}e(z,a)\leq e'(y,f(a))$. By Proposition \ref{modify}, $\sqcup S_a=f(a)$ and so
$F(a)=f(a)$ for all $a\in P$.
\end{proof}

\begin{lem}\label{3.1}
Suppose that $(P;e)$ is an $L$-ordered set and $S:P\ra L$ is a map such that $\sqcup S$
{\em ($\sqcap S$)} exists. Then for each $x\in P$, $S(x)=1$
implies that $x\leq \sqcup S$ {\em($\sqcap S\leq x$)}.
\end{lem}

\begin{proof}
Let $a=\sqcup S$ and $u\in P$ such that $S(u)=1$. Then by Theorem \ref{join and meet}, $e(a,x)=\bigwedge_{y\in P}\big( S(y)\ra e(y,x)\big)$ for all $x\in P$, hence $S(y)\leq e(y,a)$ for all $y\in P$. It follows that
$1=S(u)\leq e(u,a)$ and so $u\leq a$. The proof of the other part is similar.
\end{proof}

Suppose that $(P;e)$ is an $L$-ordered set and $f:P\ra P$ is a map. Define three maps
$S_f:P\ra L$, $T_f:P\ra L$ and $M_f:P\ra L$ by $S_f(x)=e(f(x),x)$, $T_f(x)=e'(x,f(x))$ and
$M_f(x)=S_f(x)\wedge T_f(x)$ for all $x\in P$. Moreover, by $Fix(f)$ we denote the set of all fixpoints of $f$, that is,
$Fix(f)=\{x\in P|\ f(x)=x\}$ and every point $x \in Fix(f)$ is said to be a {\it fixpoint} of $f$.

Consider the assumptions of Theorem \ref{Example}. Let $Hom(P,Q)$ be the set of all monotone maps from $P$ to $Q$. Clearly,
$(Q^P;\varepsilon')$ and $(Hom(P,Q);\varepsilon')$ are $L$-ordered sets, where $\varepsilon'(\alpha,\beta)=
\bigwedge_{x\in P}e'(\alpha(x),\beta(x))$ for all $\alpha,\beta\in Q^P$. It can be easily seen that, the map,
$\phi:Q^P\ra Hom(P,Q)$, sending $f$ to $F$
(see the notations in Theorem \ref{Example}) is a monotone map and $Fix(\phi)=Hom(P,Q)$.

\begin{thm}\label{3.2}
Let $(P;e)$ be an $L$-complete lattice and $f:(P;e)\ra (P;e)$ be a monotone map.
Then $\sqcap S_f$ and $\sqcup T_f$ are fixpoints of $f$. Indeed, $\sqcap S_f$ is the least fixpoint and
$\sqcup T_f$ is the greatest fixpoint of $f$.
\end{thm}

\begin{proof}
Let $\sqcup T_f=a$ and $\sqcap S_f=b$. Then $e(a,x)=\bigwedge_{y\in P}\big(T_f(y)\ra e(y,x)\big)$ and
$e(x,b)=\bigwedge_{y\in P}\big(S_f(y)\ra e(x,y)\big)$ for all $x\in P$ and so
$T_f(y)\leq e(y,a)$ for all $x\in P$. Since $f$ is a monotone map, then $e(y,a)\leq e(f(y),f(a))$ and hence
$T_f(y)\leq e(y,f(y))\wedge e(f(y),f(a))\leq e(y,f(a))$ for all $y\in P$. Thus by Theorem \ref{join and meet},
$e(a,f(a))=\bigwedge_{y\in P}\big(T_f(y)\ra e(y,f(a) \big)=1$. Also, $1=e(a,f(a))\leq e(f(a),f(f(a)))=T_f(f(a))$, so
by Lemma \ref{3.1}, $f(a)\leq a$. Therefore, $f(a)=a$ and $a$ is a fixpoint of $f$. Now, let $u$ be another fixpoint of $f$, then $1=e(u,f(u))=T_f(u)$, hence by Lemma \ref{3.1}, $u\leq a$, whence $a$ is the greatest
fixpoint of $f$. By a similar way, we can show that $b$ is the least fixpoint of $f$.
\end{proof}

\begin{cor}\label{cor: 3.2}
Let $(P;e)$ be an $L$-complete lattice and $f:(P;e)\ra (P;e)$ be a monotone map. Then
$\max T_f$ and $\min S_f$ exist and $\max T_f=\sqcup T_f=\sqcup M_f$ and $\min S_f=\sqcap S_f=\sqcap M_f$.
\end{cor}

\begin{proof}
By Theorem \ref{3.2} and Proposition \ref{modify}, it can be easily obtained that
$\max T_f=\sqcup T_f$ and $\min S_f=\sqcap S_f$. Let $a=\min S_f$ and $b=\max T_f$.
By Theorem \ref{3.2}, $M_f(a)=1=M_f(b)$. Also, for all $y\in P$, we have
$M_f(y)\leq S_f(y)\leq e(y,a)$ and $M_f(y)\leq T_f(y)\leq e(y,a)$, so by definition,
$\min M_f=a$ and $\max M_f=b$. Now, from \cite[Prop. 3.16]{YSI} we conclude that
$\sqcup M_f=b$ and $\sqcap M_f=a$.
\end{proof}

By \cite[Thm. 2.29]{ZXF}, we know that, if $(X;e)$ is an $L$-complete lattice and $f:X\ra X$ is a monotone map, then
$Fix(f)$ is an $L$-complete lattice. In the next theorem, we will show that each $L$-complete lattice is of this form.
That is, any $L$-complete lattice is $L$-isomorphic to $Fix(f)$ for some suitable monotone map $f$ on a suitable
$L$-complete lattice.

\begin{thm}\label{thm: rep}
Let $(P;e)$ be an $L$-complete lattice. Define $f:(L^P,\tilde{e})\ra (L^P,\tilde{e})$ by $f(S)=\downarrow \sqcup S$ for each
$S\in L^P$. Then $f$ is monotone and there exists an $L$-order isomorphism between
$(Fix(f);\tilde{e})$ and $(P;e)$.
\end{thm}

\begin{proof}
By \cite[Thm. 2.29]{ZXF}, $(L^P;\tilde{e})$ is an $L$-complete lattice. First, we show that $f$ is monotone
(clearly, $f$ is well defined). Let $S,T\in L^P$ and $x \in P$. Then by Theorem \ref{join and meet},
\begin{eqnarray}
\label{res1}e(\sqcup S,x)=\bigwedge_{y\in P}\big(S(y)\ra e(y,x) \big), \quad e(\sqcup T,x)=\bigwedge_{y\in P}\big(T(y)\ra e(y,x) \big)
\end{eqnarray}
and so
\begin{eqnarray*}
\tilde{e}(\downarrow \sqcup S,\downarrow \sqcup T)&=&\bigwedge_{y\in P}\big((\downarrow \sqcup S)(y)\ra (\downarrow \sqcup T)(y) \big)
= \bigwedge_{y\in P}\big(e(y,\sqcup S)\ra e(y,\sqcup T)\big)\\
&=& e(\sqcup S,\sqcup T), \mbox{ by Proposition \ref{YSI Prop 3.7}} \\
&\geq &\bigwedge_{y\in P}\big(S(y)\ra e(y,\sqcup T) \big), \mbox{ by (\ref{res1})}.
\end{eqnarray*}
Also, by (\ref{res1}), for all $y\in P$, $T(y)\leq e(y,\sqcup T)$, so
$S(y)\ra T(y)\leq S(y)\ra e(y,\sqcup T)$ for all $y\in P$
which implies that $\tilde{e}(S,T)=\bigwedge_{y\in P}\big(S(y)\ra T(y)\big)\leq \bigwedge_{y\in P}\big(S(y)\ra e(y,\sqcup T)\big)$.
By summing up the above results, it follows that $\tilde{e}(S,T)\leq \tilde{e}(\downarrow \sqcup S,\downarrow \sqcup T)$.
That is, $f:(L^P,\tilde{e})\ra (L^P,\tilde{e})$ is monotone. Define
$\alpha:P\ra Fix(f)$ by $\alpha(x)=\downarrow x$ for all $x\in P$. We know that $(Fix(f);\tilde{e})$ is an $L$-complete
lattice. Clearly, $\alpha$ is one-to-one. Put $S\in Fix(f)$. Then $\downarrow\sqcup S=f(S)=S$, so $S\in Im(\alpha)$.
That is, $\alpha$ is onto. Moreover, by Proposition \ref{YSI Prop 3.7}, for all $a,b\in P$,
$\tilde{e}(\alpha(a),\alpha(b))=\tilde{e}(\downarrow a,\downarrow b)=\bigwedge_{y\in P}\big(e(y,a)\ra e(y,b)\big)=e(a,b)$.
Therefore, $\alpha$ is an $L$-order isomorphism.
\end{proof}

We know that if $(P;\leq)$ is a complete lattice and $f,g:P\ra P$ are two ordered preserving maps such that
$F(x)\leq G(x)$ for all $x\in P$, then $\mu_f\leq \mu_g$, where $\mu_f$ and $\mu_g$ are
the least fixpoints of $f$ and $g$, respectively (see \cite[Section 8]{davey}).
In the next theorem, we generalize this result for $L$-complete lattices.

\begin{thm}\label{3.3} {\rm (Monotonicity rule.)}
Let $(P;e)$ be an $L$-complete lattice and $f,g:(P;e)\ra (P;e)$ be monotone maps such that
$\sqcap S_f=a$ and $\sqcap S_g=b$. Then $\bigwedge_{y\in P} e(f(y),g(y))\leq e(a,b)$.
\end{thm}

\begin{proof}
Since $\sqcap S_f=a$ and $\sqcap S_g=b$, then for each $x\in P$, we have
\begin{eqnarray}
\label{R1}  e(x,a)=\bigwedge_{y\in P}\big( e(f(y),y)\ra e(x,y)\big),\\
\label{R2}  e(x,b)=\bigwedge_{y\in P}\big( e(g(y),y)\ra e(x,y)\big).
\end{eqnarray}
By (\ref{R2}), $e(a,b)=\bigwedge_{y\in P}\big( e(g(y),y)\ra e(a,y)\big)$. Also, by (\ref{R1}),
$e(f(y),y)\leq e(a,y)$ for all $y\in P$, so $e(g(y),y)\ra e(f(y),y)\leq e(g(y),y)\ra e(a,y)$ for
all $y\in P$ and hence,
$e(a,b)=\bigwedge_{y\in P}\big( e(g(y),y)\ra e(a,y)\big)\geq \bigwedge_{y\in P}\big(e(g(y),y)\ra e(f(y),y)\big)$.
Now, we claim that $e(g(y),y)\ra e(f(y),y)\geq e(f(y),g(y))$ for all $y\in P$.
In order to show that our claim is true, it suffices to prove that
$e(f(y),g(y))\wedge e(g(y),y)\leq e(f(y),y)$, which clearly hold by (E3). Hence, our claim is true and so
$e(a,b)=\bigwedge_{y\in P}\Big( e(g(y),y)\ra e(f(y),y)\Big)\geq \bigwedge_{y\in P}e(f(y),g(y))$.
\end{proof}

\begin{thm}\label{3.5} {\rm (Rolling rule.)}
Let $(P;e)$ and $(Q;e')$ be $L$-complete lattices and $f:(P;e)\ra (Q;e')$, $g:(Q;e')\ra (P;e)$ be monotone maps.
Then the following hold:
\begin{itemize}
\item[{\rm (i)}] $g(\sqcap S_{f\circ g})=\sqcap S_{g\circ f}$.
\item[{\rm (ii)}] $\sqcap (g(S_{f\circ g}))=g(\sqcap S_{g\circ f})$.
\end{itemize}
\end{thm}

\begin{proof}
(i) The proof of this part follows from Theorem \ref{3.2}, and the Rolling Rule from \cite[8.29]{davey}.

(ii) Let $\sqcap (g(S_{f\circ g}))=b$ and $a=\sqcap S_{g\circ f}$. Then for each $x\in P$,
\begin{eqnarray}
e(x,b)=\bigwedge_{y\in P}\big(g(S_{f\circ g})(y)\ra e(x,y) \big)&=&
\bigwedge_{y\in P}\big((\bigvee_{\{z\in Q|\ g(z)=y\}}S_{f\circ g}(z))\ra e(x,y) \big)\\
&=& \bigwedge_{y\in P} \bigwedge_{\{z\in Q|\ g(z)=y\}}\big(S_{f\circ g}(z)\ra e(x,y) \big)\\
&=& \label{R3} \bigwedge_{y\in Q}\big( S_{f\circ g}(y)\ra e(x,g(y))\big).
\end{eqnarray}
\begin{eqnarray}
\label{R4} e(x,a)=\bigwedge_{y\in P}\big(S_{g\circ f}(y)\ra e(x,y)\big)=\bigwedge_{y\in P}\big(e(g\circ f(y),y)\ra e(x,y)\big).
\end{eqnarray}
Since $g$ is monotone, by (\ref{R3}) and (\ref{R4}), for each $x\in Q$,
$e(g(x),b)=\bigwedge_{y\in Q} \big(S_{f\circ g}(y)\ra e(g(x),g(y)) \big)\geq
\bigwedge_{y\in Q} \big(S_{f\circ g}(y)\ra e(x,y) \big)=e(x,\sqcap S_{f\circ g})$.
Hence, $e(g(\sqcap S_{f\circ g}),b)\geq e(\sqcap_{f\circ g},\sqcap_{f\circ g})=1$ and so by (i),
$a=\sqcap S_{g\circ f}=g(\sqcap S_{f\circ g})\leq b$.
Moreover, by (\ref{R3}), $e(f\circ g(y),y)=S_{f\circ g}(y)\leq e(b,g(y))$ for all $y\in Q$, so
$e\big(f\circ g(f(a)),f(a)\big)\leq e(b,g(f(a)))$. By Theorem \ref{3.2}, we have $g(f(a))=a$
(since $a=\sqcap S_{gof})$) and $f\circ g(f(a))=f(a)$. It follows that
$1=e(f(a),f(a))=e\big(f\circ g(f(a)),f(a)\big)=e(b,a)$.
Therefore, $a=b$, and so by (i), the proof of this part is completed.
\end{proof}

\begin{thm}\label{3.6} {\rm (Fusion rule.)}
Let $(P;e)$ and $(Q;e')$ be two $L$-complete lattices and let $f:P\ra Q$ possess a right adjoint $f':Q\ra P$.
Let $g:P\ra P$ and $h:Q\ra Q$ be monotone. Then
\begin{itemize}
\item[{\em(i)}]  $\bigwedge_{y\in P}e'(f\circ g(y),h\circ f(y))\leq e'(f(\sqcap S_g),\sqcap S_h)$.
\item[{\em(ii)}]  $e'(h\circ f(\sqcap S_g),f\circ g(\sqcap S_g))\leq e'(\sqcap S_h,f(\sqcap S_g))$.
\end{itemize}
\end{thm}

\begin{proof}
(i) By Theorem \ref{Galois}, $(f,f')$ is a fuzzy  Galois connection between $P$ and $Q$, hence
\begin{eqnarray*}
1=e(f'(\sqcap S_h),f'(\sqcap S_h))&=& e'(f(f'(\sqcap S_h)),\sqcap S_h)\\
&\Rightarrow & 1=e'\big(h(f(f'(\sqcap S_h))),h(\sqcap S_h)\big), \mbox{ since $h$ is monotone }\\
&\Rightarrow & 1=e'\big(h(f(f'(\sqcap S_h))),\sqcap S_h\big), \mbox{ by Theorem \ref{3.2}. }
\end{eqnarray*}
It follows that
\begin{eqnarray*}
e'(f\circ g(f'(\sqcap S_h)),h\big(f(f'(\sqcap S_h))\big))&=&
e'(h\big(f(f'(\sqcap S_h))\big),\sqcap S_h)\wedge e'(f\circ g(f'(\sqcap S_h)),h\big(f(f'(\sqcap S_h))\big))\\
&\leq & e'(f\circ g(f'(\sqcap S_h)),\sqcap S_h)=e(g(f'(\sqcap S_h)),f'(\sqcap S_h)).
\end{eqnarray*}
By Theorem \ref{join and meet},
for each $y\in P$, $e(g(y),y)=S_g(y)\leq e(\sqcap S_g,y)$, so
$$e(g(f'(\sqcap S_h)),f'(\sqcap S_h))\leq e(\sqcap S_g,f'(\sqcap S_h))=e'(f(\sqcap S_g),\sqcap S_h).$$
Therefore,
$\bigwedge_{y\in P}e'(f\circ g(y),h\circ f(y))\leq e'(f(\sqcap S_g),\sqcap S_h)$.

(ii) By Theorem \ref{join and meet}(ii), we know that $S_h(y)\leq e'(\sqcap S_h,y)$ for all $y\in Q$, and so by Theorem \ref{3.2},
$e'(\sqcap S_h,f(\sqcap S_g))\geq S_h(f(\sqcap S_g))=e'(h(f(\sqcap S_g)),f(\sqcap S_g))=e'(h(f(\sqcap S_g)),f(g(\sqcap S_g)))$.
\end{proof}

Note that, in Theorem \ref{3.6}, we showed that
$e'(f\circ g(f'(\sqcap S_h)),h\big(f(f'(\sqcap S_h))\big))\leq e'(f(\sqcap S_g),\sqcap S_h)$.

\begin{cor}\label{3.7} {\rm (Exchange rule.)}
Let $(P;e)$ and $(Q;e')$ be $L$-complete lattices and $f,g:P\ra Q$ and $h:Q\ra p$ be monotone maps. If
$f$ possesses a right adjoint $f':Q\ra P$, then
$$e'\Big(f\circ h\circ g(f'(\sqcap S_{g\circ h})),g\circ h\circ f(f'(\sqcap S_{g\circ h}))\Big)\leq
e'(\sqcap S_{f\circ h},\sqcap S_{g\circ h}).$$
\end{cor}

\begin{proof}
Let $h':Q\ra Q$  and $g':P\ra P$ be defined by $h'=g\circ h$ and $g'=h\circ g$. By Theorem \ref{3.6},
$$
e'\Big(f\circ g'(f'(\sqcap S_{h'})),h'\circ f(\sqcap S_{h'})\Big)\leq e'(f(\sqcap S_{g'}),\sqcap S_{h'})=
e'(f(\sqcap S_{h\circ g}),\sqcap S_{g\circ h})
$$
and by Theorem \ref{3.5}(i),
$\sqcap S_{h\circ g}=h(\sqcap S_{g\circ h})$, so that
$e'(f(\sqcap S_{h\circ g}),\sqcap S_{g\circ h})=e'(f(h(\sqcap S_{g\circ h})),\sqcap S_{g\circ h})=
S_{f\circ h}(\sqcap S_{g\circ h})$. Also, by Theorem \ref{join and meet}, for each $y\in Q$,
$S_{f\circ h}(y)\leq e'(\sqcap S_{f\circ h},y)$, so
$e'\Big(f\circ g'(f'(\sqcap S_{h'})),h'\circ f(\sqcap S_{h'})\Big)\leq e'(\sqcap S_{f\circ h},\sqcap S_{g\circ h})$.
\end{proof}


\section{{\bf Fuzzy $dcpo$s}}

In this section, we define the concept of a $t$-fixpoint and prove that if $(P;e)$ is a fuzzy $dcpo$, then the set of monotone maps on $(P;e)$ is a fuzzy $dcpo$. This will serve us in order to find some of the $t$-fixpoints of $f$. Finally, we find conditions
under which $\sqcap S_f$ exists.

\begin{thm}\label{thm 4.1}
Let $(P;e)$ be a fuzzy $dcpo$ and $H_P$ be the set of all monotone maps on $(P;e)$. Then
$(H_P;\overline{e})$ is a fuzzy $dcpo$, where $\overline{e}(f,g)=\bigwedge_{x\in P}e(f(e),g(x))$ for all $f,g\in H_P$.
\end{thm}

\begin{proof}
It is easy to see that $(H_P;\overline{e})$ is an $L$-ordered set. Let $S:H_P\ra L$ be a fuzzy directed subset of $(H_P;\overline{e})$.
Then $\bigvee_{f\in H_P}S(f)=1$ and for each $f,g\in H_P$,
\begin{eqnarray}
\label{dcpo1} S(f)\wedge S(g)\leq \bigvee_{\gamma \in H_P}\Big( S(\gam)\wedge \overline{e}(f,\gam)\wedge \overline{e}(g,\gam)\Big).
\end{eqnarray}
First, we show that $\sqcup S$ exists. That is, there exists a map $\alpha_0:P\ra P$ such that
$$
\overline{e}(\alpha_0,f)=\bigwedge_{\gam\in H_P}\Big(S(\gam)\ra \overline{e}(\gam,f)\Big) \quad \mbox{ for all $f\in H_P$, }
$$
which is equivalent to
\begin{eqnarray*}
\overline{e}(\alpha_0,f)= \bigwedge_{\gam\in H_P}\Big(S(\gam)\ra (\bigwedge_{y\in P} e\big(\gam(y),f(y)\big))\Big)
&=& \bigwedge_{\gam\in H_P}\bigwedge_{y\in P}\Big(S(\gam)\ra e\big(\gam(y),f(y)\big)\Big)\\
&=& \bigwedge_{y\in P}\bigwedge_{\gam\in H_P}\Big(S(\gam)\ra e(\gam(y),f(y))\Big).
\end{eqnarray*}
So, it suffices to show that
\begin{eqnarray}
\label{dcpo2} \bigwedge_{y\in P}e(\alpha_0(y),f(y))=
\bigwedge_{y\in P}\bigwedge_{\gam\in H_P}\Big(S(\gam)\ra e(\gam(x),f(x))\Big) \mbox{ for all $f\in H_P$ }.
\end{eqnarray}
Put $f\in H$. We claim that, for all $y\in P$, there exists an element $u_y\in P$ such that
\begin{eqnarray}
\label{dcpo3} e(u_y,f(y))=\bigwedge_{\gam\in H_P}\Big(S(\gam)\ra e(\gam(x),f(x))\Big).
\end{eqnarray}
Let $x\in P$ and $X:=\{\gam(x)|\gam\in H_P \}$. Define the map $T_x:P\ra L$ by
\begin{equation*}
T_x(y)=\left\{\begin{array}{ll}
0 & \text{$y\in P-X$ } \\
\bigvee \{S(h)|\ h\in H_P,\ h(x)=y\} & \text{$y\in X$. }\\
\end{array} \right.
\end{equation*}
Clearly, $T_x$ is a well-defined map. In the following, we show that $T_x$ is a fuzzy directed subset on $(P;e)$.

{\rm(i)}  $\bigvee_{u\in P}T_x(u)=\bigvee_{u\in X}T_x(u)=\bigvee_{\gam\in H_p}T_x(\gam(x))=
\bigvee_{\gam \in H_P}\bigvee_{\{h\in H_P|\ h(x)=\gam(x) \}}S(h)=\bigvee_{h\in H_P}S(h)=1$ (since $S:H_P\ra L$ is a fuzzy
directed set on $(H_P;\overline{e})$).

{\rm(ii)} Let $u,v\in P$. If $u\in P-X$ or $v\in P-X$, then by definition, $T_x(u)\wedge T_x(v)=0$ and so
$T_x(u)\wedge T_x(v)\leq \bigvee_{z\in P}\big(T_x(z)\wedge e(u,z)\wedge e(v,z)\big)$. Otherwise, $u=\gam_1(x)$ and
$v=\gam_2(x)$ for some $\gam_1,\gam_2\in H_P$. It follows that
\begin{eqnarray}
\label{dcpo4} T_x(u)\wedge T_x(v)=\bigvee_{\{h\in H_P|\ h(x)=\gam_1(x)\}}\bigvee_{\{k\in H_P|\ k(x)=\gam_2(x)\}} (S(h)\wedge S(k)).
\end{eqnarray}
Also, we have
\begin{eqnarray*}
\bigvee_{z\in P}\big(T_x(z)\wedge e(u,z)\wedge e(v,z) \big)&=&\bigvee_{z\in X}\big(T_x(z)\wedge e(u,z)\wedge e(v,z) \big)\\
&=& \bigvee_{\gam\in H_P}\big(T_x(\gam(x))\wedge e(u,\gam(x))\wedge e(v,\gam(x)) \big)\\
&=& \bigvee_{h\in H_P}\Big(\big(\bigvee_{\{h\in H|\ h(x)=\gam(x)\}}S(h) \big)\wedge e(u,\gam (x))\wedge e(v,\gam (x)) \Big)\\
&=& \bigvee_{h\in H_P}\bigvee_{\{h\in H|\ h(x)=\gam(x)\}} \big(S(h)\wedge e(u,\gam(x))\wedge e(v,\gam(x)) \big)\\
&=& \bigvee_{\gam\in H_P} \big(S(h)\wedge e(u,\gam(x))\wedge e(v,\gam(x))\big)\\
&=& \bigvee_{\gam\in H_P} \big(S(h)\wedge e(\gam_1(x),\gam(x))\wedge e(\gam_2(x),\gam(x))\big)\\
&\geq & \bigvee_{\gam\in H_P} \big(S(h)\wedge \overline{e}(\gam_1,\gam)\wedge \overline{e}(\gam_2,\gam)\big)\\
\end{eqnarray*}
so by (\ref{dcpo1}), $\bigvee_{z\in P}\big(T_x(z)\wedge e(u,z)\wedge e(v,z) \big)\geq S(\gam_1)\wedge S(\gam_2)$.
Since $\gam_1$ and $\gam_2$ are arbitrary elements of $H_P$ such that $\gam_1(x)=u$ and $\gam_2(x)=v$, then
$\bigvee_{z\in P}\big(T_x(z)\wedge e(u,z)\wedge e(v,z) \big)\geq S(h)\wedge S(k)$ for all $k,h\in H_P$ such that
$h(x)=u$ and $k(x)=v$, thus by (\ref{dcpo4}),
$\bigvee_{z\in P}\big(T_x(z)\wedge e(u,z)\wedge e(v,z) \big)\geq
\bigvee_{\{h\in H_P|\ h(x)=\gam_1(x)\}}\bigvee_{\{k\in H_P|\ k(x)=\gam_2(x)\}} (S(h)\wedge S(k))=
T_x(u)\wedge T_x(v)$.

(i) and (ii) imply that $T_x$ is a fuzzy directed set on $(P;e)$ for all $x\in P$, whence
by the assumption, $\sqcup T_x$ exists for all $x\in P$. Let $u_x:=\sqcup T_x$ for all $x\in P$. Then for each $x\in P$, we have
\begin{eqnarray}
\label{dcpo5} e(u_x,z)=\bigwedge_{t\in P}\big(T_x(t)\ra e(t,z) \big),\quad \mbox{ for all $z\in P$}.
\end{eqnarray}
Define a map $\alpha_0:P\ra L$ by $\alpha_0(x)=\sqcup T_x$ for all $x\in P$.
By (\ref{dcpo5}), for each $f\in H_P$ and $y,z\in P$,
\begin{eqnarray}
e(\alpha_0(y),z)=e(u_y,z)&=& \bigwedge_{t\in P}\big(T_{y}(t)\ra e(t,z) \big)\\
&=& \bigwedge_{\gam\in H_P}\big(T_{y}(\gam(y))\ra e(\gam(y),z) \big), \mbox{ by definition of $T_y$ }\\
&=& \bigwedge_{\gam\in H_P}\Big(\big(\bigvee_{\{h\in H|\ h(y)=\gam(y)\}}S(h)\big)\ra e(\gam(y),z) \Big)\\
&=& \bigwedge_{\gam\in H_P} \bigwedge_{\{h\in H|\ h(y)=\gam(y)\}} \Big( S(h)\ra e(\gam(y),z) \Big) \\
\label{dcpo6} &=& \bigwedge_{\gam\in H_P}\big(S(\gam)\ra e(\gam(y),z) \big).
\end{eqnarray}
It follows that
$\bigwedge_{y\in P}e(\alpha_0(y),f(y))=\bigwedge_{y\in P}\bigwedge_{\gam\in H_P}\big(S(\gam)\ra e(\gam(y),f(y)) \big)$ and so
(\ref{dcpo2}) holds. Hence $\alpha_0=\sqcup S$. Now, we show that $\alpha_0:P\ra P$ is a monotone map.
Let $x,y\in P$. Also, by  (\ref{dcpo6}), for all $z\in P$,
$$e(\alpha_0(x),z)=\bigwedge_{\gam\in H_P}(S(\gam)\ra e(\gam(x),z)), \quad
e(\alpha_0(y),z)=\bigwedge_{\gam\in H_P}(S(\gam)\ra e(\gam(y),z)).$$
so,
\begin{eqnarray}
\label{dcpo7} S(\gam)\leq e(\gam(y),\alpha_0(y))\\
\label{dcpo8} e(\alpha_0(x),\alpha_0(y))=\bigwedge_{\gam\in H_P}(S(\gam)\ra e(\gam(x),\alpha_0(y))).
\end{eqnarray}
Hence, by (\ref{dcpo8}),
\begin{eqnarray*}
e(x,y)\leq e(\alpha_0(x),\alpha_0(y))&\Leftrightarrow & e(x,y)\leq S(\gam)\ra e(\gam(x),\alpha_0(y)),
\mbox{ for all $\gam\in H_P$ }\\
&\Leftrightarrow & e(x,y)\wedge S(\gam)\leq  e(\gam(x),\alpha_0(y)), \mbox{ for all $\gam\in H_P$ }\\
&\Leftrightarrow & S(\gam)\leq e(x,y)\ra e(\gam(x),\alpha_0(y)), \mbox{ for all $\gam\in H_P$ }.
\end{eqnarray*}
Since for each $\gam\in H_P$, $e(x,y)\leq e(\gam(x),\gam(y))$, then by Proposition \ref{YSI Prop 3.7},
$e(\gam(y),\alpha_0(y))\leq e(\gam(x),\gam(y))\ra e(\gam(x),\alpha_0(y)) \leq e(x,y)\ra e(\gam(x),\alpha_0(y))$.
Thus by (\ref{dcpo7}), $S(\gam)\leq e(x,y)\ra e(\gam(x),\alpha_0(y))$. Therefore, $\alpha_0$ is monotone and, whence,
it belongs to $H_P$.
\end{proof}

\begin{thm}\label{thm 4.2}
Let $(P;e)$ be an fuzzy $dcpo$, $H_P$ be the set of all monotone maps on $(P;e)$ and $S:H_P\ra L$ be defined
by $S(\alpha)=\overline{e}(Id_P,\alpha)$ for all $\alpha\in H_P$. Then $\sqcup S$ exists and belongs to $H_P$. Moreover, $S(\alpha)\wedge S(\sqcup S)\leq \overline{e}(\alpha \circ \sqcup S,\sqcup S)\wedge \overline{e}(\sqcup S,\alpha\circ \sqcup S)$
for all $\alpha\in H_P$.
\end{thm}

\begin{proof}
First, we show that $S$ is a fuzzy directed subset of $(H_P,\overline{e})$.
Since $Id_P\in H$ and $S(Id_P)=\overline{e}(Id_P,Id_P)=1$, then $\bigvee_{\gam\in H_P}S(\gam)=1$.
Now, we show that $S(f)\wedge S(g)\leq\bigvee_{\alpha\in H_P}\big(S(\alpha)\wedge \overline{e}(f,\alpha)\wedge \overline{e}(g,\alpha)\big)$ for all $f,g\in H_P$. Put $f,g\in H_P$.
\begin{eqnarray}
\label{thm4.2.1} S(f)\wedge S(g)=\big(\bigwedge_{x\in P}e(x,f(x))\big)\wedge \big(\bigwedge_{x\in P}e(x,g(x))\big)=\bigwedge_{x\in P}\big(e(x,f(x))\wedge e(x,g(x))\big).
\end{eqnarray}
{\small
\begin{eqnarray}
\label{thm4.2.2} \bigvee_{\alpha\in H_P}\!\!\!\!\!\big(S(\alpha)\wedge \overline{e}(f,\alpha)\wedge \overline{e}(g,\alpha)\big)\!\!\!\!\!&=&\!\!\!\!\!
\bigvee_{\alpha\in H_P}\!\!\!\!\!\Big(\big(\bigwedge_{x\in P}e(x,\alpha(x))\big)\wedge \big(\bigwedge_{x\in P}e(f(x),\alpha(x)) \big)\wedge \big(\bigwedge_{x\in P}e(g(x),\alpha(x)) \big) \Big)\\
\label{thm4.2.3} \!\!\!\!\!&=&\!\!\!\!\! \bigvee_{\alpha\in H_P}\bigg(\Big(\bigwedge_{x\in P}e(x,\alpha(x))\wedge e(f(x),\alpha(x)) \Big)\wedge \Big(\bigwedge_{x\in P}e(g(x),\alpha(x)) \Big) \bigg).
\end{eqnarray} }
Clearly, $f\circ g\in H_P$ and for each $x\in P$, we have

(i) $e(x,f\circ g(x))\geq e(x,f(x))\wedge e(f(x),f\circ g(x))\geq e(x,f(x))\wedge e(x,g(x))$ (since $f$ is monotone).

(ii) $e(f(x),f\circ g(x))\geq e(x,g(x))$, (since $f$ is monotone),%

\noindent
which imply that $\bigwedge_{x\in P}\big(e(x,f\circ g(x))\wedge e(f(x),f\circ g(x))\big)\geq \bigwedge_{x\in P} \big(e(x,f(x))\wedge e(x,g(x)) \big)$. Also, $$\bigwedge_{x\in P} e(g(x),f\circ g(x))\geq \bigwedge_{x\in P}e(x,f(x)), \quad \mbox{ since $Im(g)\s P$ }$$
so, we have
\begin{eqnarray*}
\bigwedge_{x\in P}\Big(e(x,f\circ g(x))\wedge e(f(x),f\circ g(x))\Big)&\wedge& \bigwedge_{x\in P} e(g(x),f\circ g(x))\\
&\geq&  \bigwedge_{x\in P} \big(e(x,f(x))\wedge e(x,g(x)) \big)\wedge \bigwedge_{x\in P}e(x,f(x)) \\
&=&\bigwedge_{x\in P} \big(e(x,f(x))\wedge e(x,g(x)) \big).
\end{eqnarray*}
From $f\circ g\in H$, (\ref{thm4.2.1}) and (\ref{thm4.2.3}), it follows that
\begin{eqnarray*}
S(f)\wedge S(g)&=&\bigwedge_{x\in P} \big(e(x,f(x))\wedge e(x,g(x)) \big) \\
&\leq& \bigvee_{\alpha\in H_P}\Big(\big(\bigwedge_{x\in P}e(x,\alpha(x))\wedge e(f(x),\alpha(x)) \big)\wedge \big(\bigwedge_{x\in P}e(g(x),\alpha(x)) \big) \Big)\\
&=&\bigvee_{\alpha\in H_P}\!\!\!\!\!\big(S(\alpha)\wedge \overline{e}(f,\alpha)\wedge \overline{e}(g,\alpha)\big).
\end{eqnarray*}
Therefore, $S$ is a fuzzy directed subset of $(H_P,\overline{e})$.
By Theorem \ref{thm 4.1}, there exists $\beta\in H_P$ such that
$\beta=\sqcup S$, hence by Theorem \ref{join and meet}(i), for each
$f\in H_P$, $\overline{e}(\beta,f)=\bigwedge_{\alpha\in H_{P}}\big(S(\alpha)\ra \overline{e}(\alpha,f) \big)$,
whence $S(\alpha\circ \beta)\leq \overline{e}(\alpha\circ\beta,\beta)$. From (E2), it can be easily obtained
that $S(\alpha\circ \beta)=\overline{e}(Id_P,\alpha\circ\beta)\geq\overline{e}(Id_P,\alpha)\wedge
\overline{e}(\alpha,\alpha\circ\beta)\geq \overline{e}(Id_P,\alpha)\wedge \overline{e}(Id_P,\beta)=S(\alpha)\wedge
S(\beta)$. Thus, $S(\alpha)\wedge S(\beta)\leq \overline{e}(\alpha\circ\beta,\beta)$.
On the other hand, $ \overline{e}(\beta,\alpha\circ\beta)\geq \overline{e}(Id_P,\alpha)=S(\alpha)$
(since $Im(\beta)\s P$). By summing up the above results, we get that
$S(\alpha)\wedge S(\beta)\leq \overline{e}(\beta,\alpha\circ\beta)\wedge \overline{e}(\alpha\circ\beta,\beta)$.
\end{proof}

\begin{defn}
Let $(P;e)$ be an $L$-ordered set, $t\in L$ and $f:(P;e)\ra (P;e)$ be monotone.
An element $x\in P$ is called a {\it $t$-fixpoint} of $f$ if $t\leq e(x,f(x))\wedge e(f(x),x)$.
Obviously, the concepts of a $1$-fixpoint and a fixpoint are the same.
\end{defn}

\begin{cor}\label{Cor 4.3}
Consider the assumptions of Theorem {\rm \ref{thm 4.2}} and let $\beta=\sqcup S$. Then
\begin{itemize}
\item[{\rm (i)}] For each $x\in P$ and each $f\in H_P$, $\beta(x)$ is a $t$-fixpoint of $f$,
where $t=S(f)\wedge S(\beta)$.
\item[{\rm (ii)}] For each $f\in H_P$, $S(f)\leq \overline{e}(f,\beta\circ F)$.
\item[{\rm (iii)}] For each $f\in H_P$, there exists $u\in P$ such that $S(f)\wedge S(\beta)\leq
e(f(u),u)\wedge e(u,f(u))$.
\item[{\rm (iv)}] If $f\in H_P$ such that $S(f)=1$, then $f\circ \beta=f$. That is, $\beta(x)$ is a
fixpoint for $f$ for all $x\in P$.
\end{itemize}
\end{cor}

\begin{proof}
(i) The proof is a straightforward consequence of Theorem \ref{thm 4.2}(i).

(ii) Let $f\in H_P$. Since $\beta=\sqcup S$, then
\begin{eqnarray*}
S(f)&=&\overline{e}(Id_P,f)\leq \overline{e}(\beta,\beta\circ f), \mbox{ since $\beta$ is monotone}\\
&=& \bigwedge_{\alpha\in P}\big(S(\alpha)\ra \overline{e}(\alpha,\beta\circ f) \big), \mbox{ by Theorem \ref{join and meet}(i).}
\end{eqnarray*}
So, $S(f)\leq S(f)\ra \overline{e}(f,\beta\circ f)$, which implies that $S(f)\leq \overline{e}(f,\beta\circ f)$
(since $L$ is a frame, $a\leq a\ra b$ implies that $a=a\wedge (a\ra b)=a\wedge b$, so $a\leq b$).

(iii) By (i), for each $x\in P$, the element $u=\beta(x)$ satisfies the condition
$S(f)\wedge S(\beta)\leq e(f(u),u)\wedge e(u,f(u))$.

(iv) Let $f\in H_P$ such that $S(f)=1$. Then by Theorem \ref{join and meet}(i), $1=S(f)\leq \overline{e}(f,\beta)$, hence $f(x)\leq \beta(x)$ for all $x\in P$, which implies that
$S(\beta)=\bigwedge_{x\in P}e(x,\beta(x))\geq \bigwedge_{x\in P}e(x,f(x))=S(f)=1$.
So, by (i), $1=S(f)\wedge S(\beta)\leq \overline{e}(f\circ\beta,\beta)\wedge \overline{e}(\beta,f\circ\beta)$. That
is, $f\circ \beta=\beta$.
\end{proof}

We know that if $(P;\leq)$ is a CPO and $f:P\ra P$ is an ordered preserving map, then $Fix(f)$
has a least element. Moreover, by Theorem \ref{3.2}, if $(P;e)$ is an $L$-complete lattice, then $\sqcap S_f$
exists and is the least fixpoint of $f$. In the sequel, we attempt to find conditions for a monotone map
on a fuzzy $dcpo$ $(P;e)$ under which $\sqcap S_f$ exists.

\begin{defn}
Let $(P;e)$ be an $L$-ordered set, $S\in L^P$ and $X\s P$. An element $b\in X$ is called a {\it join of $S$ in $X$} and
is denoted by $\bigsqcup_{X} S=b$ if, for each $x\in X$, $e(b,x)=\bigwedge_{y\in X}\big(S(y)\ra e(y,x) \big)$.
In a similar way, we can define the notion of a meet of $S$ in $X$, $\bigsqcap_{X} S$.
\end{defn}

\begin{rmk}
Let $(P;e)$ be an $L$-ordered set, $S\in L^P$ such that $a=\sqcup S$ and $X\s P$. Suppose that $b$ is a join of $S$ in $X$. Then
$b\in X$ and for each $x\in X$,
$e(b,x)=\bigwedge_{y\in X}\big(S(y)\ra e(y,x) \big)\geq \bigwedge_{y\in P}\big(S(y)\ra e(y,x) \big)=e(a,x)$ and so
$1=e(b,b)=e(a,b)$. Thus, $a\leq_e b$. By a similar way, we can show that, if $a'=\sqcap S$ and $b'$ is a meet of $S$ in $X$, then
$b'\leq_e a'$.
\end{rmk}

In a special case, if $X$ is a subset of an $L$-ordered set $(P;e)$,
$S\in L^P$, $b=\bigsqcup_{X}S$ ($b'=\bigsqcap_{X}S$)
and  $a=\sqcup S$ ($a'=\sqcap S$) such that
$Supp(S):=\{x\in P|\ S(x)\neq 0\}$ is a subset of $X$, then for all $x\in X$, we have
$$e(a,x)=\bigwedge_{y\in P}\big(S(y)\ra e(y,x) \big)=\bigwedge_{y\in X}\big(S(y)\ra e(y,x) \big)=e(b,x).$$
So, $e(a,b)=e(b,b)=1$. That is, $a\leq_e b$. A similar proof shows that $b'\leq_e a'$.

\begin{defn}
Let $X$ be a non-empty subset of an $L$-ordered set $(P;e)$ and $a\in P$. An element $b\in X$ is called a
{\it strong $L$-cover} for $a$ in $X$ if $e(a,x)=e(b,x)$ for all $x\in X$.
We must note that, from $b\in X$ it follows that
$e(a,b)=e(b,b)=1$. Also, if $a$ has a strong $L$-cover
in $X$, then it is unique. Indeed, if $b,b'\in X$ are strong $L$-covers for $a$ in $X$, then
by definition, $1=e(b',b')=e(a,b')=e(b,b')$. Similarly, $e(b',b)=1$, so $b=b'$.
\end{defn}

\begin{prop}\label{prop 4.4}
Let $X$ be a non-empty subset of an $L$-ordered set $(P;e)$, $S\in L^P$ $a=\sqcup S$ such that
$Supp(S)\s X$.
Then $b$ is a strong $L$-cover for $a$ in $X$ if and only if $b=\bigsqcup_{X}S$.
\end{prop}

\begin{proof}
Let $b=\bigsqcup_{X}S$. Then for all $x\in P$,
\begin{eqnarray*}
e(b,x)&=&\bigwedge_{y\in X}\big(S(y)\ra e(y,x)\big)=\bigwedge_{y\in P}\big(S(y)\ra e(y,x)\big), \mbox{ since $Supp(S)\s X$ }\\
&=& e(a,x), \mbox{ since $\sqcup S=a$}.
\end{eqnarray*}
It follows that $b$ is a strong $L$-cover for $a$ in $X$. The proof of the converse is similar.
\end{proof}

\begin{defn}
Let $(P;e)$ be an $L$-ordered set and $f:P\ra P$ be monotone. An element $S\in L^P$ is called
{\it $f$-invariant} if,  for all $x\in P$, $S(x)\leq S(f(x))$.
\end{defn}

\begin{exm}\label{exm 4.5}
Let $(P;e)$ be an $L$-ordered set, $f:P\ra P$ be a monotone map and $a$ be a fixpoint of $f$.

(i) Then the maps $\downarrow a:P\ra L$ and $\uparrow a:P\ra L$ defined by $\downarrow a(x)=e(x,a)$ and
$\uparrow a(x)=e(a,x)$ (see \cite[Def. 3.15]{YSI}), respectively are $f$-invariant. In fact,
for each $x\in P$, $\downarrow a(x)=e(x,a)\leq e(f(x),f(a))=e(f(x),a)=\downarrow a(f(x))$.
So $\downarrow a$ is $f$-invariant. By a similar way,
$\uparrow a$ is $f$-invariant.

(ii) The maps $S_f$ and $T_f$ are $f$-invariant. Let $x\in P$. Since $f$ is monotone, then
$$S_f(x)=e(f(x),x)\leq e(f(f(x)),f(x))=S_f(f(x)),\quad T_f(x)=e(x,f(x))\leq e(f(x),f(f(x)))=T_f(f(x)).$$
Therefore, $S_f$ and $T_f$ are $f$-invariant.
\end{exm}

\begin{thm}\label{thm 4.6}
Let $(P;e)$ be a fuzzy $dcpo$ with zero (that is, there is $0\in P$ such that $e(0,x)=1$
for all $x\in P$), $f\in H_P$, $Y=\{a\in P|\ e(a, f(a))=1\}$ and $M=Fix(f)$. If each element of $Y$ has a
strong $L$-cover in $M$, then $(M;e)$ is a fuzzy sub $dcpo$ of $(P;e)$.
\end{thm}

\begin{proof}
Clearly, $(P;\leq_e)$ is a CPO, so by \cite[Thm. 8.22]{davey}, $f$  has a fixpoint (or $M$ has a least element).
Assume that each element of $Y$ has a strong $L$-cover. Put a fuzzy directed subset $S$ of $(M;e)$.
Define $\overline{S}:P\ra L$, by $\overline{S}(x)=S(x)$ for all $x\in M$ and $\overline{S}(x)=0$ for
all $x\in P-M$. It can be easily shown that $\overline{S}$ is a fuzzy directed subset of $(P;e)$. It follow that
$\sqcup \overline{S}$ exists. Let $a=\sqcup\overline{S}$. Then for each $x\in P$,
\begin{eqnarray}
\label{thm4.6.1} e(a,x)=\bigwedge_{y\in P}\big(\overline{S}(y)\ra e(y,x) \big)=\bigwedge_{y\in M}\big(S(y)\ra e(y,x) \big)
\end{eqnarray}
and so $1=e(a,a)=\bigwedge_{y\in M}\big(S(y)\ra e(y,a) \big)\leq \bigwedge_{y\in M}\big(S(y)\ra e(f(y),f(a)) \big)=
\bigwedge_{y\in M}\big(S(y)\ra e(y,f(a)) \big)=e(a,f(a))$, which implies that $a\in Y$. Hence by the assumption,
$a$ has a strong $L$-cover in $M$, $b$ say, whence $e(a,x)=e(b,x)$ for all $x\in M$.
From (\ref{thm4.6.1}) it follows that $e(b,x)=e(a,x)=\bigwedge_{y\in M}\big(S(y)\ra e(y,x) \big)$. Therefore,
$b=\bigsqcap_{X}S$ and so $(M;e)$ is a fuzzy $dcpo$ with zero.  Conversely,
let $(M;e)$ be a fuzzy $dcpo$. Put $a\in Y$. Define $S:P\ra L$ by $S(a)=1$ and $S(x)=0$ for all $x\in P-\{a\}$.
\end{proof}

\begin{exm}
Let $(P;\leq)$ be a CPO. It can be easily shown that the $\{0,1\}$-ordered set $(P;e_\leq)$ is a
fuzzy $dcpo$ with zero. Put a monotone map $f:(P;e_\leq)\ra (P;e_\leq)$.
Since $f$ is monotone, then for each $a\in P$ satisfying the condition $a\leq f(a)$, we have
$A=\{x\in P|\ a\leq x\}$ is a CPO and $f:A\ra A$ is a monotone map, so by  \cite[Thm. 8.22]{davey} $f$ has
a least fixpoint on $A$, $b$ say. Let $x$ be another fixpoint of $f$.

(1) If $a\leq x$, then $b\leq x$, so $e_\leq(a,x)=e_\leq(b,x)=1$.

(2) If $a\nleq x$, then $b\nleq x$, so $e_\leq(a,x)=e_\leq(b,x)=0$.

Hence, $b$ is a strong $L$-cover for $a$ in $Fix(f)$, whence $(P;e_\leq)$  satisfies the
conditions of Theorem \ref{thm 4.6}.
\end{exm}




\begin{thebibliography}{99}
{\footnotesize

\bibitem{Bel1} R. B\v{e}lohl\'{a}vek, ``Fuzzy Relational Systems: Foundations and Principles", {\it Kluwer Acad. Publ.}, New York, 2002.


\bibitem{Bel2} R. B\v{e}lohl\'{a}vek, Concept lattices and order in fuzzy logic, {\it Ann. Pure Appl. Logic} {\bf 128} (2004), 277--298.


\bibitem{Blyth} T. S. Blyth, ``Lattices and Ordered Algebraic Structures",  {\it Springer-Verlag}, London, 2005.

\bibitem{F1} C. Coppola, G. Gerla, T. Pacelli, Convergence and fixed points by fuzzy orders,
{\it Fuzzy Sets Syst.} {\bf 159} (2008), 1178--1190.


\bibitem{davey} B.A. Davey, H.A. Priestley, ``Introduction to Lattices and Order", {\it
Cambridge University Press}, Second edition, 2002.

\bibitem{DA} A. Davis, A characterization of complete lattices, {\it Pacific J. Math.} {\bf 5} (1955), 311--319.


\bibitem{Fan} L. Fan, A new approach to quantitative domain theory, {\it Electronic Notes  Theor. Comp. Sci.}
{\bf 45} (2001), 77--87.

\bibitem{Fuentes} R. Fuentes-Gonz\'{a}lez, Down and up operators associated to fuzzy relations and t-norms:
A definition of fuzzy semi-ideals, {\it Fuzzy Sets Syst.} {\bf 117} (2001), 377--389.


\bibitem{Johnston} P.T. Johnstone, ``Stone Spaces", {\it Cambridge University Press}, Cambridge, 1982.


\bibitem{LaiZhang} H. Lai, D. Zhang, Complete and directed complete $\Omega$-categories, {\it Theor. Computer Sci.}, {\bf 388} (2007), 1--25.





\bibitem{TA} A. Tarski, A lattice theoretical fixed point theorem and its applications,
{\it Pacific J. Math.} {\bf 5} (1955), 285--309.



\bibitem{Wagner} K.R. Wagner, Solving recursive domain equations with enriched categories, Ph.D. Thesis, Carnegie Mellon University, Technical Report CMU-CS-94-159, July 1994.



\bibitem{Yao-Shi} W. Yao, F. G. Shi, Quantitative domains via fuzzy sets: Part II: Fuzzy Scott topology
on fuzzy directed-complete posets, {\it Fuzzy Sets Syst.} {\bf 173} (2011), 60--80.

\bibitem{Y2} W. Yao, An approach to fuzzy frames via fuzzy posets, {\it Fuzzy Sets Syst.} {\bf 166} (2011), 75--89.

\bibitem{YL} W. Yao, L. X. Lu, Fuzzy Galois connections on fuzzy posets, {\it Math. Log. Quart.} {\bf 55} (2009), 105--112.

\bibitem{YSI} W. Yao, Quantitative domains via fuzzy sets: Part I: Continuity of fuzzy directed complete posets,
{\it Fuzzy Sets Syst.} {\bf 161} (2010), 973--987.


\bibitem{Y2S} W. Yao, F. G. Shi, Quantitative domains via fuzzy sets: Part II: Fuzzy Scott topology
on fuzzy directed-complete posets, {\it Fuzzy Sets Syst.} {\bf 173} (2011), 60--80.


\bibitem{ZF} Q. Y. Zhang, L. Fan, Continuity in quantitative domains, {\it Fuzzy Sets Syst.} {\bf 154} (2005), 118--131.

\bibitem{ZL} D. X. Zhang, Y. M. Liu, $L$-fuzzy version of Stones representation theorem for distributive lattices,
{\it Fuzzy Sets Syst.} {\bf 76} (1995), 259--270.

\bibitem{ZF2} Q. Y. Zhang, W. X. Xie, Section-retraction-pairs between fuzzy domains, {\it Fuzzy Sets Syst.} {\bf 158} (2007), 99--114.

\bibitem{ZXF} Q. Y. Zhang, W. X. Xie, L. Fan, Fuzzy complete lattices, {\it Fuzzy Sets Syst.} {\bf 160} (2009), 2275--2291.


}
\end{thebibliography}
\end{document}